\documentclass[11pt]{amsart}
\usepackage{amssymb,amsmath,amscd,amsthm}
\usepackage{graphicx}
\usepackage[table,dvipsnames]{xcolor}
\usepackage[verbose]{wrapfig}
\usepackage[font=small,labelfont=bf]{caption}
\usepackage{mathtools}
\usepackage{tikz}
\usepackage{subfigure}
\usepackage[margin=1in]{geometry}
\usepackage{hyperref}
\usepackage[capitalise]{cleveref}
\usepackage{longtable}
\usepackage{booktabs}

\usepackage{tikz-cd}
\usepackage{graphicx}
\usepackage{float}
\usepackage{algorithm}%
\usepackage{algorithmicx}%
\usepackage{algpseudocode}%

\usepackage{silence}
\usepackage{caption}

\usepackage{bm}
\usepackage{longtable}
\usepackage{multicol}

\hypersetup{
    colorlinks=true,
    linkcolor=blue,
    filecolor=blue,      
    }

\makeatletter
\def\l@subsection{\@tocline{2}{0pt}{2.5pc}{5pc}{}}
\def\l@subsubsection{\@tocline{2}{0pt}{5pc}{7.5pc}{}}
    
\newtheorem*{MT}{Main Theorem}
\newtheorem{theorem}{Theorem}[section]
\newtheorem{corollary}[theorem]{Corollary}
\newtheorem{proposition}[theorem]{Proposition}
\newtheorem{lemma}[theorem]{Lemma}
\newtheorem*{ack}{Acknowledgements}
\newtheorem{remark}[theorem]{Remark}

\theoremstyle{definition}
\newtheorem{definition}[theorem]{Definition}

\numberwithin{equation}{section}

\makeatletter
\newtheorem*{rep@theorem}{\rep@title}
\newcommand{\newreptheorem}[2]{%
\newenvironment{rep#1}[1]{%
 \def\rep@title{#2 \ref{##1}}%
 \begin{rep@theorem}}%
 {\end{rep@theorem}}}
\makeatother

\newreptheorem{theorem}{Theorem}

\DeclareMathOperator{\rank}{rank}
\newcommand{\floor}[1]{\left\lfloor #1 \right\rfloor}

\newcommand{\CircMatrix}[1]{\text{circ}\{#1\}}
\newcommand{\IsotRed}[2]{\mathfrak{c}_{#1}^{#2}}
\newcommand{\IsotExp}[1]{\hat{[\partial_{#1}]}}
\newcommand{\IsotExpRow}[1]{[\partial_{#1}]^\text{row}}
\newcommand{\IndicesOfOrbit}[3]{{I_{#3}^{#1,#2}}}
\newcommand{\GBoundary}[1]{[\partial'_{#1}]^{S_\ell,T}}

\algrenewcommand\algorithmicrequire{\textbf{Input:}}
\algrenewcommand\algorithmicensure{\textbf{Output:}}



\begin{document}

\title{Computing homology of $\mathbb{Z}_k$-complexes from their quotients}

\author{Christine Escher}
\address{Department of Mathematics, Oregon State University, Corvallis, OR 97331}
\email{christine.escher@oregonstate.edu}

\author{Chad Giusti}
\address{Department of Mathematics, Oregon State University, Corvallis, OR 97331}
\email{chad.giusti@oregonstate.edu}

\author{Chung-Ping Lai}
\address{Department of Mathematics, Oregon State University, Corvallis, OR 97331}
\email{laich@oregonstate.edu }

\date{\today}

\begin{abstract}
In this paper, we investigate the question of how one can recover the homology of a simplicial complex $X$ equipped with a regular action of a finite group $G$ from the structure of its quotient space $X/G.$ Specifically, we describe a process for enriching the structure of the chain complex $C_\ast(X/G; \mathbb{F})$ using the data of a complex of groups, a framework developed by Bridson and Corsen for encoding the local structure of a group action. We interpret this data through the lens of matrix representations of the acting group, and combine this structure with the standard simplicial boundary matrices for $X/G$ to construct a surrogate chain complex. In the case $G = \mathbb{Z}_k,$ the group ring $\mathbb{F}G$ is commutative and matrices over $\mathbb{F}G$ admit a Smith normal form, allowing us to recover the homology of $G$ from this surrogate complex. This algebraic approach complements the geometric compression algorithm for equivariant simplicial complexes described by Carbone, Nanda, and Naqvi.
\end{abstract}

\maketitle


\section{Introduction}

In recent years, the concept of homology has gained traction in data analysis. The majority of applications of homology to applied problems involve the computation of homology for a simplicial complexes, which are often natural to define from data and thus can be directly interpreted in terms of model systems. However, even for modest data sets the resulting complexes are often inconveniently large, and so it is natural ask whether structure in the complex systems can be used to reduce the size of the computations needed to recover homology. One common kind of structure is symmetry, encoded as an action of a group $G$ on the complex. The study of equivariant homology of simplicial $G$-complexes is well-trod ground in algebraic topology \cite{BredonCompactTransformationGroups}, but many questions that are natural in the context of applications remain open. 

In this paper, inspired by recent work of Carbone, Nanda, and Naqvi \cite{CarboneNandaNaqvi}, we ask what information about a group action one must retain in order to reconstruct the homology of a finite regular $G$-complex from the homology of its quotient complex under the action of $G$. The necessary framework is that of a \emph{complex of groups}, independently developed by Haefliger \cite{Haefliger91} and Corson \cite{Corson92}, which generalizes Bass-Serre theory \cite{bass1993covering}. Combining this data with a suitable morphism, one can reconstruct the combinatorial structure of a finite simplicial $G$-complex \cite{BridsonHaefliger, carbone2013reconstructing}.  In \cite{CarboneNandaNaqvi} this idea is applied to give an explicit algorithm for compressing simplicial $G$-complexes \cite{CarboneNandaNaqvi}. Here, we develop an algebraic analog of this compression process in the context of finite cyclic group actions on a finite simplicial complex. This allows us to enrich the standard simplicial chain complex for the quotient complex using the data of the action to produce a surrogate chain complex over the group ring $\mathbb{F}G$, and, when $\mathbb{F}G$ admits a Smith normal form, from this data recover the $\mathbb{F}$-homology of the original $G$-complex. Our main result can be stated as follows,

\begin{MT} Let $\mathbb{F}$ be a field, $X$ be a regular $\mathbb{Z}_k$-complex, $\ell: X/\mathbb{Z}_k \to X$ a lift, and $d$ a non-negative integer. Let $\partial_d$ be the $d$-th boundary map in the chain complex $C_\ast(X; \mathbb{F}).$ Let $([\psi'_0],\dots,[\psi'_{n}])$ and $([\omega'_0],\dots,[\omega'_{m}])$ respectively be any $d$-th and a $(d-1)$-th ordered simplicial basis of $X/\mathbb{Z}_k$ and $[\partial'_d]^{S_\ell,T}$ be the $d$-th $\mathbb{Z}_k$-boundary matrix $[\partial_d']^{S_\ell,T}$ corresponding to $X$, $\ell$, and the ordered simplicial bases. Let $\alpha$ be any generator of $\mathbb{Z}_k$ and $\rho_\alpha$ the representation of $\mathbb{Z}_k$ corresponding to $\alpha$.  If $D$ be a smith normal form of $[\partial'_d]^{S_\ell,T}$, then $\rank (\partial_d) = \sum\limits_{i} \rank (\rho_\alpha(D_{i,i}))$.
\end{MT}

The paper is organized as follows. In \cref{sec: Preliminaries} we recall  simplicial $G$-complexes, complexes of groups, and results about the algebra of group rings. In \cref{sec: Compatible Matrix} we develop a framework for selecting ordered bases of the chain complex of $X/G$ that induce boundary matrices which will support the reconstruction of the homology of $X$. In \cref{sec: G-Boundary Matrix} we describe a procedure for compressing the data of the group action into an \emph{isotropy transfer triple}, a variant of a complex of groups which contains the additional information necessary for the homology reconstruction.  We also introduce the $G$-boundary matrix, an algebraic combination of the boundary matrix with this data. Finally, in \cref{sec: Rank Algorithm} we select a suitable representation of $\mathbb{Z}_k$ and relate the $G$-boundary matrix and the expanded boundary matrices introduced in \cref{sec: Compatible Matrix}. We use this relationship to state and prove the result stated above, \cref{Thm: Main}. Readers primarily interested in the main result but not the proof can bypass \cref{sec: Compatible Matrix} and \cref{Lemma: Expansion equals Representation}.

\begin{ack} A portion of this work draws from the PhD thesis of C.-P. Lai. C. Escher acknowledges support from the Simons Foundation (\#585481, C. Escher). C. Giusti is partially supported by the Air Force Office of Scientific Research under award number FA9550-21-1-0266.

\end{ack}

\section{Preliminaries}\label{sec: Preliminaries}

Throughout the paper, we assume all groups and simplicial complexes are finite. For a simplicial complex $X = (V, \Sigma),$ by abuse we will write $\psi \in X$ to mean $\psi \in \Sigma.$ All of our simplicial complexes will be \emph{oriented}, meaning that each simplex will be assigned a choice of ordering of its vertices up to parity.

\subsection{Simplicial \texorpdfstring{$G$}{G}-Complexes}
 Recall that a simplicial complex $X$ equipped with an action of a finite group $G$ is called a \emph{simplicial $G$-complex}, which we will shorten to \emph{$G$-complex} when unambiguous. 
 To orient the reader and set notation, in this section we will recall some basic notions and set notation regarding $G$-complexes and their quotients from Chapter 3 of Bredon \cite{BredonCompactTransformationGroups}.

Fix a finite group $G$ and let $X$ be a $G$-complex. For each simplex $\psi = \{v_0,\dots,v_d\}$ of $X$ and each group element $g\in G$,  write $g\psi := \{gv\mid v\in \psi\} = \{gv_0,\dots,gv_d\}$ for the image of $\psi$ under the $G$-action. 
To ensure that the orbit space of a $G$-complex $X$ is again a simplicial complex, additional assumptions must be made about $X$. 

\begin{definition}\label{Def:RegularSimplicialGComplex} Given a $G$-complex $X = (V, \Sigma)$, the action of $G$ on $X$ is \textbf{regular} and $X$ is called a \textbf{regular $G$-complex} if for every subgroup $H$ of $G$, for all $h_0,h_1,\dots,h_d \in H$, if $\{v_0,\dots,v_d\}$ and $\{h_0v_0,\dots,h_d v_d\}$ are both simplices of $X$, then there exists $h\in H$ such that $h(v_i) =h_i(v_i)$ for all $i=0,1,\dots,d$. That is, $h\{v_0,\dots,v_d\} =\{h_0v_0,\dots,h_d v_d\}$.
\end{definition}

Note that in the definition of regularity, the vertices $v_0,\dots,v_d$ need not be distinct. 

\begin{proposition}\label{Prop:RegularConditionA}
Let $G$ be a group and $X$ a  regular $G$-complex. 
For any $g\in G$ and $v \in X$, if $v$ and $gv$ belong to the same simplex $\psi \in X$, then $v=gv$.  That is, $g$ leaves $\psi \cap g\psi$ point-wise fixed for all $g \in G$.
\end{proposition}

Not every $G$-complex is regular, but we can always construct a regular complex through successive barycentric subdivisions.  

\begin{proposition} Let $X$ be a $G$-complex and $B(X)$ its barycentric subdivision. Then the  induced action of $G$ on $B(B(X))$ is regular.
\end{proposition}

The regularity condition on a $G$-complex $X$ is precisely the information necessary to define a simplicial complex structure on the orbits of the simplices of $X.$ Writing $S/G$ for the $G$-orbits of the elements of a $G$-set $S$, that structure is given as follows.

\begin{proposition}\label{Prop: QuotientSimplicialComplex} Let $X=(V,\Sigma)$ be a regular $G$-complex. Then $X/G := (V',\Sigma')$, where $V' = V/G$ and $\Sigma':=\{\{Gv_0,\dots,Gv_d\}\mid \{v_d,\dots,v_d\}\in \Sigma\}$, is a simplicial complex. Moreover, the canonical orbit map $\pi :V\rightarrow V/G$ on $V$ induces a simplicial map $\pi:X\rightarrow X/G$. Finally, for any $d$-simplex $\psi'$ of $X/G$, $\pi^{-1}(\psi')$ is a $G$-orbit of $d$-simplices of $X$. In fact, the map $\Sigma'_d \rightarrow \Sigma_d/G,\,\psi'\mapsto \pi^{-1}(\psi')$ is injective. 
\end{proposition}

For a $G$-complex $X$, we define $X/G$ to be the \textbf{quotient complex of $X$} by the action of $G$ and $\pi$ to be the \textbf{quotient map.} Finally, a simplex $\psi\in \Sigma$ is then said to be \textbf{over} a simplex $\psi'\in \Sigma'$ when $\pi(\psi) =\psi'$. Throughout the article, we will denote vertices and simplices in the quotient complex $X/G = (V', \Sigma')$ using primes. 

We will also need the following properties of regular $G$-complexes.

\begin{definition}\label{Def: Isotropy subgroup of simplex} Let $X$ be a  $G$-complex and $\psi$ a simplex of $X$. Then 
$G_\psi = \{g\in G\mid g\psi =\psi\} $ is called 
 the \textbf{isotropy subgroup} of $\psi$. \footnote{In some references, the object in \cref{Def: Isotropy subgroup of simplex} is called the \emph{maximal} isotropy subgroup. See, for example, \cite{IllmanEquivariantTriangulation}.}\end{definition}

In the case of a regular $G$-complex any $g\in G_\psi$ fixes $\psi$ vertex-wise, see Proposition \ref{Prop:RegularConditionA}. 
This implies that for regular $G$-complexes many (but not all) alternative definitions of the isotropy subgroup associated to a simplex are equivalent. 

\medskip

Observe that if $X = (V, \Sigma)$ is a regular $G$-complex, and $\psi \subseteq \omega \in X$, then $G_\psi\leq G_\omega.$ Moreover, we show that if an element $g \in G$ leaves the face $\omega$ of a simplex $\psi$ invariant, that is,  $g \omega \subseteq \psi$, then $g$ must fix that face.

\begin{proposition}\label{Prop:OrbitNotAnotherFace}Let $X$ be a regular $G$-complex, and $\omega \subseteq \psi \in X$. Let $g\in G$ such that $g \omega \subseteq \psi$.  Then $g \in G_\omega.$
\end{proposition}

\begin{proof}
Let $X$ be a regular $G$-complex, $\omega\subseteq \psi \in X$ and $g\in G$. Applying $g^{-1}$ to both sides of $g\omega\subseteq\psi,$ we obtain $\omega \subseteq g^{-1}\psi$. Thus $\omega \subseteq g^{-1}\psi \cap \psi$. By  Proposition \ref{Prop:RegularConditionA}, $g^{-1}v = v$ for any $v\in g^{-1}\psi \cap \psi$, and thus $gv = v$. In particular $g\omega = \omega$, so $g \in G_\omega.$
\end{proof}

We introduce the following notation to identify individual elements in an orbit with left cosets of isotropy subgroups. The well-definedness of the definition follows from the canonical bijection $G/G_\psi \rightarrow G\psi$ between the left cosets of the isotropy subgroup in $G$ and the $G$-orbits.

\begin{definition}\label{Def: Coset act on simplex}Let $X$ be a regular $G$-complex and $\psi\in X$. Enumerate the cosets of $G_\psi$ in $G$ as $G/G_\psi = \{(gG_\psi)_1,\dots,(gG_\psi)_{|G/G_\psi|}\}$.  We define $(gG_\psi)_i\psi$ to be a simplex $h\psi \in G\psi$, where $h \in (gG_\psi)_i$.
\end{definition}

\subsection{Complexes of Groups}

Given a regular $G$-complex $X$, Bridson and Haefliger (\cite{BridsonHaefliger} Chapter III.C) describe how to recover $G$ from $X/G$ using a structure called a \emph{complex of groups} $\mathcal{A}$ over $X/G,$ along with a morphism $\Phi$ from $\mathcal{A}$ to a suitable \emph{constant} complex over $X/G.$ In \cite{CarboneNandaNaqvi}, this framework is reformulated to describe an algorithm that compresses a regular $G$-complex $X$ of data that is equivalent to $(\mathcal{A},\Phi)$, and another algorithm that reverses the process. In this section, we recall both constructions.

The construction of a complex of groups leverages the fact that for any group action elements in the same orbit have conjugate isotropy subgroups. We use language from category theory \cite{ferreira2006pseudo,fiore2011euler} to describe this process.

\begin{definition}\label{Def: 2-morphism of group category}Let $G_0,G_1$ be groups and $\mathfrak{f}_1,\mathfrak{f}_2:G_0\rightarrow G_1$ be homomorphisms. A \textbf{$2$-morphism} $g: \mathfrak{f}_1 \Rightarrow \mathfrak{f}_2$ is a group element $g\in G_1$ that relates $\mathfrak{f}_1$ and $\mathfrak{f}_2$ by conjugation, that is
$ \mathfrak{f}_1(h) = g\mathfrak{f}_2(h)g^{-1}\text{ for all }  h \in G_0\,.$
\end{definition}

Variants of the following definitions appeared simultaneously
in the work of Haefliger (\cite{Haefliger91}) and Corson (\cite{Corson92}). We will follow the version in \cite{CarboneNandaNaqvi} with slightly modified notations.

\begin{definition}\label{Def: Complex of Groups} Let $Y$ be a simplicial complex. A \textbf{complex of groups} $\mathcal{A} = (\mathfrak{G},\mathfrak{f},\mathfrak{g})$ \textbf{over} $Y$ consists of:

\begin{itemize}
\item [(1)] for each $\psi\in Y$ a finite group $\mathfrak{G}_\psi$, 

\item [(2)] for each pair $\psi_2 \subseteq \psi_1 \in Y$ there is an injective homomorphism $\mathfrak{f}_{\psi_1\psi_2} : \mathfrak{G}_{\psi_1}\rightarrow \mathfrak{G}_{\psi_2}$, and

\item [(3)] for each triple $\psi_3 \subseteq \psi_2 \subseteq \psi_1\in Y$ there is a $2$-morphism $\mathfrak{g}_{\psi_1 \psi_2 \psi_3}$ : $\mathfrak{f}_{\psi_2 \psi_3 }\circ \mathfrak{f}_{\psi_1 \psi_2} \Rightarrow \mathfrak{f}_{\psi_1\psi_3}$,
\end{itemize}
such that for any four simplices $\psi_4 \subseteq \psi_3 \subseteq \psi_2 \subseteq \psi_1 \in Y$, the following constraints hold:
\begin{itemize}
    \item [(a)] $\mathfrak{f}_{\psi_1\psi_1}$ is the identity $\mathfrak{G}_{\psi_1} \rightarrow \mathfrak{G}_{\psi_1}$.
    \item [(b)]  $\mathfrak{g}_{\psi_1\psi_1\psi_2} = \mathfrak{g}_{\psi_1\psi_2\psi_2}:\mathfrak{f}_{\psi_1\psi_2}\Rightarrow \mathfrak{f}_{\psi_1\psi_2}$ are both given by $e_{\psi_2},$ the identity element of $\mathfrak{G}_{\psi_2}$.
    \item [(c)] The following cocycle condition holds in the group $G_{\psi_4}$: $$\mathfrak{f}_{\psi_3\psi_4}(\mathfrak{g}_{\psi_1\psi_2\psi_3})\cdot \mathfrak{g}_{\psi_1\psi_3\psi_4} = \mathfrak{g}_{\psi_2\psi_3\psi_4}\cdot \mathfrak{g}_{\psi_1\psi_2\psi_4}.$$
\end{itemize}

\end{definition}

There is a notion of morphisms between complexes of groups, as discussed in \cite{BridsonHaefliger}. In this article, we only require a specific class of morphisms, which will allow us to streamline our discussion of morphisms following \cite{CarboneNandaNaqvi}. We first define the common target of this class of morphisms.

\begin{definition} Let $Y$ be a simplicial complex and $G$ be a group. The \textbf{constant $G$-valued complex of groups} $\mathcal{G}$ over $Y$ is the complex of groups $(\mathfrak{G}^G,\mathfrak{f}^G,\mathfrak{g}^G)$ over $Y$ defined a follows:
\begin{itemize}
\item [(1)] for each simplex $\psi \in Y$, $\mathfrak{G}^G_\psi = G$,
\item [(2)] for each pair $\psi_2 \subseteq \psi_1 \in Y$, $\mathfrak{f}^G_{\psi_1\psi_2} = id_G$, and
\item [(3)] for each triple $\psi_3 \subseteq \psi_2 \subseteq \psi_1\in Y$,  $\mathfrak{g}^G_{\psi_1 \psi_2 \psi_3}$ is the identity element of $G$.
\end{itemize}
\end{definition}

\begin{definition}\label{Def: Constant G Morphism} Let $Y$ be a simplicial complex, $G$ a finite group, and  $\mathcal{A}= (\mathfrak{G},\mathfrak{f},\mathfrak{g})$ a complex of groups over $Y$. We say $\Phi$ is an \textbf{injective morphism from $\mathcal{A}$ to the constant $G$-valued complex of groups}, (or simply a \textbf{constant $G$-morphism from $\mathcal{A}$}) if $\Phi$ consists of the following data.
\begin{itemize}
\item [(1)] for each simplex $\psi \in Y$, an injective group homomorphism $\Phi_\psi: \mathfrak{G}_\psi \rightarrow G$,
\item [(2)] for each pair $\psi_2 \subseteq \psi_1 \in Y$, a $2$-morphism $\Phi(\psi_1,\psi_2):\Phi_{\psi_2}\circ \mathfrak{f}_{\psi_1\psi_2}\Rightarrow \Phi_{\psi_1}$.
\item [(3)] for each triple $\psi_3 \subseteq \psi_2 \subseteq \psi_1\in Y$, we require that the following constraints hold:
\begin{itemize}
\item [(a)] $\Phi(\psi_1,\psi_1)$ is the identity element of $G$.
\item [(b)] $\Phi_{\psi_3}(\mathfrak{g}_{\psi_1\psi_2\psi_3})\cdot \Phi_{\psi_1\psi_3} = \Phi_{\psi_2\psi_3}\cdot\Phi_{\psi_1\psi_2}$.
\end{itemize}
\end{itemize}

\end{definition}

To associate a complex of groups to a regular $G$-complex, we begin by defining two functions: a \emph{lift} and a \emph{transfer}.

\begin{definition}\label{Def:Lift} Let $X=(V,\Sigma)$ be a regular $G$-complex, $X/G = (V',\Sigma')$ its quotient complex, and $\pi: X\rightarrow X/G$ the simplicial quotient map. A \textbf{lift} $\ell$ to $X$ is a function $\ell:\Sigma'\rightarrow \Sigma$ such that $\pi\circ \ell = id$. 
\end{definition} 

Note that we do not require that the image of a lift constitutes a simplicial complex, which is why we avoid the use of the term ``section'', which suggests a coherent choice of preimage. Indeed, in general, a lift does not preserve the face relations between simplices.

For the rest of this article, we assume that every regular $G$-complex $X$ is assigned a lift $\ell$.

\begin{definition}\label{Def: Unassociated transfer} Let $X$ be a simplicial complex and let $\mathcal{F}(X) = \{(\psi,\omega) \mid \omega \subseteq \psi \in X, \dim(\psi) = \dim(\omega) + 1\}$. A \textbf{transfer function} on $X$ is a function $T: \mathcal{F}(X) \rightarrow G.$
\end{definition}

In order to associate a transfer function to a regular $G$-complex, we identify group elements with certain properties based on the following observation.

\begin{proposition}\label{Prop: Unique simplex implied by face in quotient} Let $X$ be a regular $G$-complex, and $\pi : X \to X/G$. Fix $\psi \in X$, and let $\psi' = \pi (\psi)$. Then $\omega' \subseteq \psi'$ if and only if there exists a unique $\omega\in X$ such that $\omega \subseteq \psi$ and $\pi(\omega) = \omega'$.

\begin{proof}
Let $X = (V, \Sigma)$ be a regular $G$-complex and $\psi = \{v_0, \dots, v_d\} \in X$. Then $\psi' = \{Gv_0,\dots,Gv_d\}$. Suppose $\omega' \subseteq \psi'$. Then $\omega' = \{Gv_{i_1},\dots,Gv_{i_k}\}$ for some subset $\{ v_{i_1}, \dots v_{i_k}\}\,\subseteq \{ v_0, \dots, v_d\}\,.$ Thus $\omega = \{v_{i_1},\dots,v_{i_k}\} \subseteq \psi$ and $\pi(\omega) = \omega'$. To see that $\omega$ is unique, observe that if $\omega_1,\omega_2 \subseteq \psi$ and $\pi(\omega_1) = \pi(\omega_2) = \omega'$, then $\omega_1$ and $\omega_2$ are in the same $G$-orbit. Since $\omega_2 = g \omega_1 \subset \psi$ for some $g \in G$, Proposition \ref{Prop:OrbitNotAnotherFace} implies that $\omega_1 = \omega_2.$ 
The converse follows from the fact that $\pi$ preserves face relations.
\end{proof}
\end{proposition}

We now can describe our desired transfer function, which depends on the choice of a lift function.

\begin{definition}\label{Def: Associated Transfer} Let $X$ be a regular $G$-complex and $\ell$ a lift to $X$. A \textbf{transfer function associated to $X$ and $\ell$} is a function $T_\ell : \mathcal{F}(X/G) \longrightarrow G$ such that $T_\ell(\psi',\omega') := g$, where $g \in G$ such that $g^{-1}\ell(\omega')\subseteq \ell(\psi')$.
\end{definition}

The existence of $g \in G$ in Definition \ref{Def: Associated Transfer} is justified as follows. By the definition of a lift, $\pi(\ell(\psi')) =\psi'$ and $\pi(\ell(\omega')) =\omega'$. Because $(\psi',\omega')\in 
\mathcal{F}(X/G),$ $\omega' \subseteq \psi'$. Hence by Proposition \ref{Prop: Unique simplex implied by face in quotient}, there exists $\omega$ such that $ \omega\subseteq \ell(\psi')$ and $\pi(\omega) = \omega'$. In particular $\pi(\omega) = \omega' = \pi(\ell(\omega'))$ means that $\omega$ and $\ell(\omega')$ are in the same orbit, so there exists $g\in G$ such that $g\omega = \ell(\omega')$, and $g^{-1}\ell(\omega') = \omega \subseteq \ell(\psi')$.

Note that $X$ and $\ell$ do not determine $T_\ell$ uniquely. However, $X$, $\ell$, and $T_\ell$ will together be sufficient to uniquely determine our desired associated complex of groups and constant $G$-morphism.

\begin{definition}\label{Def: Associated Complex and Morphism}Let $X$ be a regular $G$-complex. Let $\ell$ be a lift to $X$ and $T_\ell$ be an associated transfer. 
\begin{itemize}
\item The \textbf{complex and morphism associated to the triple $(X,\ell,T_\ell)$} is the pair $(\mathcal{A},\Phi)$, where $\mathcal{A} = (\mathfrak{G},\mathfrak{f}, g)$ is a complex of groups over $X/G$ and $\Phi$ is a constant $G$-morphism from $\mathcal{A}$ defined as follows. 

\begin{itemize}
\item [($\mathcal{A}$1)] For each $\psi' \in X/G$, $\mathfrak{G}_{\psi'} = G_{\ell(\psi')},$ the isotropy subgroup of $\ell(\psi')$.
\item [($\mathcal{A}$2)] For each pair $\psi'_2 \subseteq \psi'_1 \in X/G$, $\mathfrak{f}_{\psi'_1\psi'_2}$ is the map $h\mapsto T_\ell(\psi'_1,\psi'_2)hT_\ell(\psi'_1,\psi'_2)^{-1}$.
\item [($\mathcal{A}$3)] For each triple $\psi_3 \subseteq \psi_2 \subseteq \psi_1 \in X/G$, the $2$-morphism $\mathfrak{g}_{\psi_1 \psi_2 \psi_3}$ is given by $T_\ell(\psi'_2,\psi'_3)\cdot T_\ell(\psi'_1,\psi'_2) \cdot T_\ell(\psi'_1,\psi'_3)$.
\item [($\Phi$1)] For each simplex $\psi'\in X/G$, $\Phi_{\psi'}$ is the inclusion map $\mathfrak{G}_{\psi'}\hookrightarrow G$.
\item [($\Phi$2)] For each pair $\psi'_2 \subseteq \psi'_1 \in X/G$, the $2$-morphism $\Phi(\psi'_1,\psi'_2)$ is given by $T_\ell(\psi'_1,\psi'_2)$.
\end{itemize}
\vspace{0.2cm}
\item  The \textbf{output triple associated to $(X,\ell,T_\ell)$} is the triple $(X/G,S_\ell,T_\ell),$ where $S_\ell$ is the function $\Sigma' \rightarrow \{\text{subgroups of $G$}\}$ given by $\psi' \mapsto G_{\ell(\psi')}.$

\end{itemize}
\end{definition}

Note that the output triple $(X/G,S_\ell,T_\ell)$ uniquely determines the complex and morphism $(\mathcal{A},\Phi)$ associated to $(X, \ell, T_\ell)$: $X/G$ is the underlying simplicial complex of $\mathcal{A}$, $S_\ell$ determines item $(\mathcal{A}1)$ and $(\Phi1)$, $T_\ell$ determines items $(\mathcal{A}2)$, $(\mathcal{A}3)$, and $(\Phi2)$.

We end this section by recalling a function from \cite{CarboneNandaNaqvi} that will be useful to us.

\begin{definition}\label{Def: CarboneNandaNaqvi coset function}Let $G$ be a finite group, $\mathcal{A}=(\mathfrak{G},\mathfrak{f},\mathfrak{g})$ a complex of groups over a simplicial complex $X$, and $\Phi$ a constant $G$-morphism from $\mathcal{A}$. For each $\omega'\in X$, let $c_{\omega'}: G\rightarrow G/\Phi_{\omega'}(\mathfrak{G}_{\omega'})$ be the canonical surjective map of sets that sends each group element $g$ to the corresponding left coset $g\Phi_{\omega'}(\mathfrak{G}_{\omega'})$. \end{definition}

Note that the function $c_\omega$ is depends only on the choice of the lift $\ell,$ since it is determined by $\Phi_{\omega'}$ and $\mathfrak{G}_{\omega'}$, both of which are determined by the lift per conditions $(\mathcal{A}1)$ and $(\Phi1)$ of Definition \ref{Def: Associated Complex and Morphism}. \

\subsection{Circulant Matrices and Group Rings}\label{subsec: Circulant Matrix and Group Rings}

In this section we recall and set notation for circulant matrices, following Kra and Simanca \cite{OnCirculantMatrices}, which will play a fundamental role in our construction once we restrict our attention to finite cyclic group. In particular, we highlight a result of Hurley \cite{GroupRingAndRingofMatrices} which relates group rings over finite cyclic groups to circulant matrices.

We take subscript indices of $v$ modulo $k$ in the following definition.

\begin{definition}\label{Def: Circulant Matrix} Let $\mathbb{F}$ be a field and $\mathbb{F}^k$ the vector space of row vectors with entries in $\mathbb{F}$. The \textbf{shift operator} on $\mathbb{F}^k$ is given by $\mathcal{S}(v_1,\dots,v_k):= (v_{k},v_1,\dots,v_{k-1})$.

The \textbf{circulant matrix} associated to the vector $v\in \mathbb{F}^k$, denoted by $\CircMatrix{v}$, is the $(k \times k)$-matrix with the $c$-th row given by $\mathcal{S}^{c-1}(v) = (v_{2-c},v_{3-c},\dots,v_{k+1-c}), \;c =1,\dots,k$.
\end{definition}

\begin{remark} Note that the transpose of a circulant matrix is also circulant. In particular, $\CircMatrix{(v_1,\dots,v_{n})}^T = \CircMatrix{(v_1,v_{n},v_{n-1},\dots,v_3,v_2)}$.
\end{remark}

Circulant matrices to provide a useful characterization of the group ring of a finite cyclic group.

\begin{definition}\label{Def: RG Matrix} Let $R$ be a ring and $G = (g_1, \dots, g_k)$ be a finite group equipped with the given ordering of its elements. A $(k\times k)$-matrix $M$ with entries in $R$ is an \textbf{$RG$-matrix over $R$} if $M_{i,j} = M_{i',j'}$ whenever $g_i^{-1}g_j = g_{i'}^{-1}g_{j'}$. The collection of all such matrices is denoted by $RGM_{k}$.
\end{definition}

The following theorem shows that mapping coefficients of generators to entries of a matrix is in fact an embedding of the group ring $RG$ to $RGM_{n}$.

\begin{theorem}[\cite{GroupRingAndRingofMatrices}, Theorem 1]\label{Thm: RG embeds to Matrix} Let $R$ be a ring and $G = (g_1, \dots, g_{k})$ be a finite group with a fixed ordering of its elements.  For $w = \sum_{i=1}^{k} a_{g_i}g_i \in RG,$ define $M(RG,w) \in RGM_{k}$ to be the matrix with $M_{i,j} = a_{g_i^{-1}g_j}.$ Then there is a ring isomorphism $\rho: RG \to RGM_{k}$ given by $\rho: w \mapsto M(RG,w)$.
\end{theorem}

Applying Theorem \ref{Thm: RG embeds to Matrix} to finite cyclic groups, we obtain the following corollary which is an essential tool for the proof of our main theorem.

\begin{corollary} \label{Prop: Cyclic gives Circulant}Let $R$ be a ring and $G = \mathbb{Z}_k = (e, \alpha,\dots,\alpha^{k-1})$ a finite cyclic group of order $k$ with the given ordering.  Then $M(RG,\sum_{i=1}^{k}a_i\alpha^{i-1}) = 
 \CircMatrix{(a_1,a_2,\dots,a_{k})}.$ That is, 

$$
M(RG,\sum_{i=0}^{k-1}a_i\alpha^i) =  \begin{bmatrix}
a_{1} & a_{2} & a_{3} & \cdots & a_{k}\\
a_{k} & a_{1} & a_{2} & \cdots & a_{k-1}\\
\vdots & \vdots & \vdots & \ddots & \vdots \\
a_{2} & a_{3} & a_{4} & \cdots & a_{1}\\
\end{bmatrix}.$$

\noindent In other words, $R\mathbb{Z}_k$ embeds into the ring of $(k\times k)$-matrices over $R$ as the set of $k\times k$ circulant matrices over $R$.
\end{corollary}

\section{Compatible Associated Matrix}\label{sec: Compatible Matrix}

We now develop the algebraic framework which underlies our procedure. Let $X$ be a simplicial complex. We denote by $\Sigma_d$ the set of $d$-simplices of $X$, so $X = (V,\sqcup_d\Sigma_d),$. We take $C_d(X;\mathbb{F})$ to be the $d$-th simplicial chain group of $X$ with coefficients in a field $\mathbb{F}$ and write  $\partial_d$ for the $d$th boundary map of this chain complex. For an oriented $d$-simplex $\sigma = \langle v_0,\dots,v_d \rangle \in \Sigma_d$, we denote by $[\sigma] = [v_0,\dots,v_d]$ the elementary $d$-chain in $C_d(X;\mathbb{F})$ corresponding to $\sigma$. Note that in such a case, if $\Sigma_d = \{\sigma_1,\dots\sigma_n\}$, then $\{[\sigma_1],\dots,[\sigma_n]\}$ is a basis of $C_d(X;\mathbb{F})$, and we call it the \textbf{simplicial basis}. If we order $\Sigma_d$, denoted by $(\sigma_1,\dots\sigma_n)$, then $([\sigma_1],\dots,[\sigma_n])$ is called an \textbf{ordered simplicial basis of $C_d(X;\mathbb{F})$}.  

Also note that there is an induced action of $G$ on $C_d(X;\mathbb{F})$ by $g[\psi] = \pm[g\psi]$ for each oriented $d$-simplex $\psi,$ extended linearly. With suitable choice of orientations, we can assume that $g[\psi] = [g\psi]$ for all $g \in G.$ The simplicial quotient map $\pi$ then induces a chain map, which we also denote $\pi$ by abuse of notation,
$$\pi : C_d(X;\mathbb{F}) \longrightarrow C_d(X/G;\mathbb{F}), [v_0,\dots,v_d] \mapsto [v_0',\dots,v_d']=[Gv_0,\dots,Gv_d],$$
on an ordered simplicial basis, extended linearly. We then call $\pi$ the \textbf{quotient chain map}. 

\subsection{Compatible Ordered Basis}

In this section our goal is to describe matrix representations of $\partial_d$ and $\partial'_d$ which are suitable for our reconstruction. 
We first orient and order $\Sigma'_d$, which induces an ordered simplicial basis of $C_d(X/G;\mathbb{F})$. We then use this basis to construct a compatible ordered simplicial basis for $C_d(X;\mathbb{F})$.

\begin{definition}\label{Def: Compatible Simplicial Basis} Let $\{[\psi']\}_{\psi' \in \Sigma_d'}$ be the simplicial basis for $C_d(X/G;\mathbb{F}).$ The simplicial basis of $C_d(X;\mathbb{F})$ \textbf{compatible with $\{[\psi']\}_{\psi'\in \Sigma'_d}$} is the basis $\{[\psi]\}_{\psi\in \Sigma_d}$ with the property that for any simplex $\psi = [v_0,\dots, v_d]\in \Sigma_d$,  $\pi([\psi])= \pi([v_0,\dots, v_d])  = [\pi v_0,\dots, \pi v_d] =  [\pi\psi]$.
\end{definition}

The next step is to construct an ordering of $\Sigma_d$.  We first note that for a regular $G$-complex $X = (V,\sqcup_d \Sigma_d)$, the  induced action of $G$ on $\Sigma_d$ partitions $\Sigma_d$ by orbits of $d$-simplices.  Let $X/G = (V', \sqcup_d \Sigma'_d)$, $\Sigma'_d = \{\psi_1,\dots,\psi_n\}$, and $\pi$ be the quotient simplicial map. By \cref{Prop: QuotientSimplicialComplex} and the surjectivity of $\pi$, $\{\pi^{-1}(\psi'_1),\dots,\pi^{-1}(\psi'_n)\}$ is precisely the list of $G$-orbits of $\Sigma_d$. Hence an ordering $(\psi'_1,\dots,\psi'_n)$ of $\Sigma'_d$ yields an ordering $(\pi^{-1}(\psi'_1), \dots ,\pi^{-1}(\psi'_n))$ of the $G$-orbits of $\Sigma_d$.

On the other hand, there is no canonical choice of an ordering of each orbit $\pi^{-1}(\psi'_i)$. We can specify an ordering by fixing an ordering of the elements of the group $G$ as well as choosing a lift.
We first observe that there is an alternative way to describe the orbits $\pi^{-1}(\psi'_i)$, once a lift $\ell$ is given.

\begin{remark}\label{Rmk: Orbit of Lift is Inverse Projection}
For any simplex $\psi'$ of $X/G$, $G(\ell(\psi')) = \pi^{-1}(\psi')$. That is, the orbit of $\ell(\psi')$ is  the orbit of simplices over $\psi'$. This follows from \cref{Def:Lift}, since $\pi(\ell(\psi')) = \psi'$ and so $\ell(\psi')$ is in the orbit $\pi^{-1}(\psi')$.
\end{remark}

Moving forward, we will assume that all finite groups $G$ are equipped with a fixed ordering $(g_1,\dots,g_k)$ of their elements, so that the first element is the identity, that is, $g_1 = e.$ Observe that once we have such an ordering, for any subgroup $H\leq G$, there is a canonical way to order the left cosets in $G/H$.

\begin{definition}\label{Def: Ordering of Cosets compatible with G}Let $G= (g_1,\dots,g_{k})$ be an ordered finite group. For any $H\leq G$, we define $((gH)_1,\dots,(gH)_{|G/H|})$, an \textbf{ordering of $G/H$ compatible with the ordering $(g_1,\dots,g_{k})$} of $G$, by removing all duplicated elements from the $\{g_1H,\dots,g_{k}H\}$ after their first appearance.  We then say that {\bf $G/H=((gH)_1,\dots,(gH)_{|G/H|})$ is ordered by $G=(g_1,\dots,g_{k})$}.
\end{definition}

For the remainder of the article we will assume that for any subgroup $H$ of an ordered finite group $G$, the left cosets $G/H$ are ordered by $G$. Note that the coset $(gH)_1 = g_1H = eH$ is then just $H$. 
 
 Lastly, consider the case where $H = G_\psi,$ the isotropy subgroup of a simplex of a regular $G$-complex.  Using the notation of \cref{Def: Coset act on simplex}, we obtain $(gG_\psi)_1\psi = (G_\psi)\psi = \psi$.  For a given lift, an ordering of $G$ therefore induces an ordering of $\pi^{-1}(\psi')$ for any $\psi' \in \Sigma'$ which is defined as follows.

\begin{definition}\label{Def: Ordering of Orbit Compatible with Group and Lift} Let $\psi'$ be a $d$-simplex of $X/G$. Let $H = G_{\ell(\psi')}$. Let $G/H = ((gH)_1,\dots,(gH)_{|G/H|})$ be ordered by $G=(g_1,\dots,g_{k})$. The \textbf{ordering of $\pi^{-1}(\psi')$ compatible with $(g_1,\dots,g_{k})$} is then the ordering
$$((gH)_1(\ell(\psi')),\dots,(gH)_{|G/H|}(\ell(\psi')))$$
of $\pi^{-1}(\psi') = G(\ell(\psi'))$. When $\pi^{-1}(\psi')$ is equipped with such an ordering, we say $\pi^{-1}(\psi')$ is \textbf{ordered by $(g_1,\dots,g_{k})$}.
\end{definition}

Now we combine these two definitions to construct an ordering of the set $\Sigma_d$ of $d$-simplices of a regular $G$-complex $X$, given a lift $\ell$ and an ordering $(\psi'_1,\dots,\psi'_{n})$ of $\Sigma'_d$. We begin by defining a finite sequence $(n_a)$ of integers for $a= 1,\dots, n$ (where $n = |\Sigma'_d|$). Denote $G_{\ell(\psi'_a)}$ by $H_a$. Then let $n_1 := 1$ and inductively define $n_{a+1} := n_a + |G/H_a|$.

For each simplex $\psi'_a$ of $\Sigma'_d$, let the simplices in the orbit
$$\pi^{-1}(\psi'_a) = G\ell(\psi'_a) = ((gH_a)_1\psi_{n_a},\dots,(gH_a)_{|G/H_a|}\psi_{n_a})$$ be ordered by $G$. Then define  
$$\psi_{n_a+k-1} := (gH_a)_{k}\ell(\psi'_a)
\,\,\text{for} \,\,k = 1,\dots,|G/H_a|\,.$$

Note that $\{n_a,\dots ,n_a + |G/H_a|-1\}$ for $a= 1,\dots, n$ is a partition of $\{1,\dots,|\Sigma_d|\}$, and each $\{n_a,\dots ,n_a + |G/H_a|-1\}$ is a set of indices of a set $\{G\ell(\psi'_a) = \pi^{-1}(\psi'_a)\}$ of $d$-simplices, which is part of the partition $\{\pi^{-1}(\psi'_a)\}_{\psi'_a\in \Sigma'_d}$ of $\Sigma_d$. Hence $\psi_1,\dots,\psi_{|\Sigma_d|}$ is a complete list of $d$-simplices of $X$ with no repetition, and we call this ordering of $\Sigma_d$ compatible with $\ell$, $(\psi'_1,\dots,\psi'_{n})$, and the ordering of $G$.

\begin{definition}\label{Def: Ordering of Simplices compatible with quotient/lift/group} The ordering $(\psi_1,\dots,\psi_{|\Sigma_d|})$ defined above is the \textbf{ordering of $\Sigma_d$ compatible with $\ell$, $(\psi'_1,\dots,\psi'_{n})$, and the ordering of $G$}.
\end{definition}

We also need to label the sets $\{n_a,\dots ,n_a + |G/H_a|-1\}$ of indices used in Definition \ref{Def: Ordering of Simplices compatible with quotient/lift/group}, whose elements are the indices of each distinct orbit of $d$-simplices.

\begin{definition}\label{Def: Partition Index by Orbits} Let $(\psi_1,\dots,\psi_{|\Sigma_d|})$ be the compatible ordering of $\Sigma_d$. Then the \textbf{$d$-th lifted partition of $\{1,\dots, |\Sigma_d|\}$ (corresponding to $X$, the ordering of $G$, and $(\psi'_1,\dots,\psi'_{n})$)} is the collection $\{\IndicesOfOrbit{X}{d}{a}\}_{a= 1,\dots, n}$ of sets of integers where  $\IndicesOfOrbit{X}{d}{a}: =\{n_a,n_a+1,\dots,n_a + |G/G_{\psi_{n_a}}|\}$ and $n_a$ is the unique index of $l(\psi'_a)$ assigned by the compatible ordering.
\end{definition}

When the context is clear, we simply call $\{\IndicesOfOrbit{X}{d}{a}\}_{a= 1,\dots, n}$ the lifted partition of indices.

Now we have all the ingredients to construct the desired ordered basis of the $d$-th chain group of a regular $G$-complex. Combining \cref{Def: Compatible Simplicial Basis} and \cref{Def: Ordering of Simplices compatible with quotient/lift/group}, we can describe our target matrix representation of the $d$-th boundary map of $X$.

\begin{definition}\label{Def: Associated Matrix Compatible with Quotient} The ordered simplicial basis $([\psi_1],\dots,[\psi_{|\Sigma_d|}])$ of $C_d(X;\mathbb{F})$ \textbf{compatible with $([\psi'_1] ,\dots, [\psi'_{n}])$, $\ell$, and the ordering of $G$} is the compatible simplicial basis of $C_d(X;\mathbb{F})$ with ordering induced by the compatible ordering of $\Sigma_d$.

The \textbf{matrix $[\partial_d]$ associated to $\partial_d$ (compatible with $([\psi'_1],\dots,[\psi'_{m}])$, $([\omega'_1],\dots, [\omega'_{n}])$, $\ell$, and the ordering of $G$)} is the matrix representation of $\partial_d$  relative to the compatible ordered simplicial basis of $C_d(X;\mathbb{F})$ and $C_{d-1}(X;\mathbb{F})$.
\end{definition}

For the rest of this article we let $[\partial'_d]$ be the matrix representation of $\partial'_d$ with respect to the ordered simplicial basis $([\psi'_1],\dots,[\psi'_{n}])$ and $([\omega'_1],\dots, [\omega'_{m}])$ of $C_d(X/G;\mathbb{F})$ and $C_{d-1}(X/G;\mathbb{F})$ respectively.

The following lemma describes the entries of a compatible associated matrix.

\begin{lemma}\label{Lemma: Entries of Compatible Matrix}
Let $[\partial_d]$ be the compatible matrix associated to $\partial_d$. Let $\psi_j \in \Sigma_d(X),\omega_i \in \Sigma_{d-1}(X)$, $\psi_b' \in \Sigma'_d(X/G),\omega_a' \in \Sigma'_{d-1}(X/G)$.
If $\pi(\psi_j) =\psi'_b$, $\pi(\omega_i)=\omega'_a$, and $\omega_i\subset \psi_j$, then 
$$[\partial_d]_{i,j} = [\partial'_d]_{a,b}\, .$$

\end{lemma}

\begin{proof}

By assumption $\pi(\psi_j) = \psi'_b$, hence the compatibility of the simplicial basis implies that $\pi[\psi_j] = [\psi'_b]$ and $\pi[\omega_i] = [\omega'_a]$. Let $\partial_d[\psi_j] = \sum_{r}c_r[\omega_r]$ and $\partial_d[\psi'_b] = \sum_{r'}c'_{r'}[\omega'_{r'}]$. Then $c_i = [\partial_d]_{i,j}$ and $c'_a = [\partial'_d]_{a,b}$. Applying $\pi$ to both sides of $\partial_d[\psi_j] = \sum_{r}c_r[\omega_r]$ we obtain $\partial'_d[\psi'_a] = \sum_{r'}(\sum_{\pi\omega_r = \omega'_{r'}}c_r)[\omega'_{r'}]$. 

By assumption $\pi(\omega_i) =\omega'_a$, hence $\omega_i\in \pi^{-1}(\omega'_a)$. Recall from \cref{Rmk: Orbit of Lift is Inverse Projection} that $\pi^{-1}(\omega'_a)$ is a $G$-orbit, and thus it is the orbit of $\omega_i$. But $\omega_i\subseteq \psi_j$ and by \cref{Prop:OrbitNotAnotherFace} no other $(d-1)$-simplex in the orbit of $\omega_i$ is a face of $\psi_j$. Hence by the construction of $\partial_d$, for any $r\neq i$ such that $\pi\omega_r = \omega'_a$, $c_r = 0$. It follows that $[\partial'_d]_{a,b} = c'_a = \sum_{\pi\omega_r = \omega'_{a}}c_r = c_i = [\partial_d]_{i,j}$, as desired.
\end{proof}

\subsection{Isotropy Expansion}\label{sec: Relationship}

In this section we use the isotropy subgroups to   expand a compatible boundary matrix of a simplicial $G$-complex to a larger matrix of the same rank. 

We define the following function between integers. Recall from \cref{Def: Ordering of Cosets compatible with G} that we have ordered $G/H$ for every subgroup $H$ of $G$. Hence for every left coset $g_cH$ of $H$ in $G$, there exists a unique integer $\gamma$ such that $(gH)_\gamma = g_cH$, where $g_c \in G = \{g_1, \dots , g_k\}$.

\begin{definition}\label{Def: Isotrope Index Reducing Function} 
Let $G$ be an ordered finite group, $|G| = k$, $X$ a regular $G$-complex with quotient $X/G$, and $\ell$ a lift to $X$. Let $d$ be a non-negative integer and $(\psi'_1,\dots,\psi'_{n})$ be an ordering of $\Sigma'_d$.  Let $\{\IndicesOfOrbit{X}{d}{b}\}_{b = 1,\dots,n}$ be the lifted partitions of simplices from \cref{Def: Partition Index by Orbits} and, for each $b$, denote by $q_b$ the smallest element in $\IndicesOfOrbit{X}{d}{b}$.  Define the function $\mathfrak{I}_\ell: \{1,\dots,n\}\times \{1,\dots,k\} \rightarrow \mathbb{Z}$ by
$$\mathfrak{I}_\ell(b,c) := q_b +1 +\gamma$$
where $\gamma$ is the unique integer that satisfies $(gG_{\ell(\psi'_{b})})_{\gamma} = g_{c}G_{\ell(\psi'_{b})}$. The \textbf{$d$-th isotropy index-reducing function of $X$ (corresponding to $\ell$, the ordering of $G$, and the ordering $(\psi'_1,\dots,\psi'_{n})$ of $\Sigma'_d$)}  is the function $\IsotRed{X}{d}:\{1,2,\dots,nk\} \rightarrow \{1,2,\dots,|\Sigma_d|\}$ defined by:
$$\IsotRed{X}{d}(i) :=  \mathfrak{I}_\ell(\floor{i/k}+1, i -\floor{i/k}k)\, .$$

\end{definition}

\vspace{0.2cm}
\begin{remark}\label{indexreducing}
If we let $i = (b-1)k + c = bk -k +c$ for integers $(b,c)\in \{1,\dots,n\}\times \{1,\dots,k\}$, and $\gamma$ be the integer satisfying $(gG_{\ell(\psi'_{b})})_{\gamma} = g_{c}G_{\ell(\psi'_{b})}$, then 
$\IsotRed{X}{d}(i) = q_b -1 +\gamma\, .$
Also note that $\sum_{j=1}^{b-1}|\IndicesOfOrbit{X}{d}{j}| = n_{b}-1$, since the elements of  $\IndicesOfOrbit{X}{d}{j}$ are consecutive integers. Hence we also have $\IsotRed{X}{d}(i) = \sum_{j=1}^{\floor{i/k}}|\IndicesOfOrbit{X}{d}{j}|+\gamma\, .$
\end{remark}

In the following proposition we prove two properties of this index-reducing function.

\begin{proposition}\label{Prop: Range of Isotrop Index Reducing}
Let $X$ be a regular $G$-complex. Let $\{\IndicesOfOrbit{X}{d}{b}\}$ be the lifted partitions of indices, and $\IsotRed{X}{d}$ the corresponding $d$-th isotropy index-reducing function of $X$. For $b= 1,\dots, n$, let $L_b$ be the set of consecutive integers $\{kb-k+1,\dots,kb\}$. Then
$\IsotRed{X}{d}(L_b) = \IndicesOfOrbit{X}{d}{b}.$
In particular,  $\IsotRed{X}{d}$ is surjective.
\end{proposition}

\begin{proof} We first prove that $\IsotRed{X}{d}(L_b) \subseteq \IndicesOfOrbit{X}{d}{b}$. Fix an integer $b$, and let $q_b$ be the index of $\ell(\psi'_b)$, then $\IndicesOfOrbit{X}{d}{b} = \{q_b, q_b+1,\dots,q_b-1+|G/G_{\ell(\psi'_{b})}|\}$. Then $\IsotRed{X}{d}(L_b) \subseteq \IndicesOfOrbit{X}{d}{b}$ follows directly from the definition of $\IsotRed{X}{d}$. For the other direction, note that each integer in $\IndicesOfOrbit{X}{d}{b}$ can be expressed as $q_b -1 +\gamma$ for a unique $\gamma\in \{ 1,\dots,|G/G_{\ell(\psi'_{b})}|\}$. We now claim that there exists an integer  $c\in \{1,\dots, k\}$ such that $(gG_{\ell(\psi'_{b})})_{\gamma} = g_{c}G_{\ell(\psi'_{b})}$. We can then let $i = bk-k+c$ which implies that $\IsotRed{X}{d}(i) = q_b-1+\gamma$ by Remark \ref{indexreducing}.
To prove the claim, let $\gamma\in \{ 1,\dots,|G/G_{\ell(\psi'_{b})}|\}$.  
Recall that by the definition of left cosets $\{g_{c}G_{\ell(\psi'_{b})}\}_{c = 1,\dots,k} = \{gG_{\ell(\psi'_{b})}\}_{g\in G}$ is an exhaustive (with potential repetition) list of elements in $G/G_{\ell(\psi'_{b})} =
(gG_{\ell(\psi'_{b})})_{\gamma}\}_{\gamma=1,\dots,|G/G_{\ell(\psi'_{b})}|}$.  But then $(gG_{\ell(\psi'_{b})})_{\gamma}$ must be an element of $ \{g_{c}G_{\ell(\psi'_{b})}\}_{c = 1,\dots,k}$, hence there exists some $c\in \{1,\dots, k\}$ such that $g_cG_{\ell(\psi'_{b})} = (gG_{\ell(\psi'_{b})})_{\gamma}$ as desired.

Lastly, to see that $\IsotRed{X}{d}$ is surjective, we note that $\{L_b\}$ partitions its domain $\{1,\dots, nk\}$, while $\{\IndicesOfOrbit{X}{d}{b}\}$ partitions its codomain $\{1,\dots, |\Sigma_d|\}$. Hence $\IsotRed{X}{d}(L_b) = \IndicesOfOrbit{X}{d}{b}$ implies that $\IsotRed{X}{d}$ is surjective.
\end{proof}

We now define the enlarged matrix.

\begin{definition}\label{Def: Isotrope Expansion} Let $\mathbb{F}$ be a field. Let $([\psi'_1],\dots,[\psi'_{n}])$ and $([\omega'_1],\dots, [\omega'_{m}])$ be ordered simplicial bases of $C_d(X/G;\mathbb{F})$ and $C_{d-1}(X/G;\mathbb{F})$ respectively. Let $[\partial_d]$ be the compatible matrix associated to $\partial_d$ from \cref{Def: Associated Matrix Compatible with Quotient}. Then, $\IsotExp{d}$, the \textbf{isotropy expansion of $[\partial_d]$}, is the $(m|G|\times n|G|)$ matrix over $\mathbb{F}$ defined as follows:
$$\IsotExp{d}_{i,j} := [\partial_d]_{\IsotRed{X}{d-1}(i),\IsotRed{X}{d}(j)}$$
for $i= 1,\dots, m|G|$ and $j = 1,\dots,n|G|$.
\end{definition}

The next lemma shows that such an expansion does not change the rank of the matrix.

\begin{lemma} \label{Lemma: Isotrope Expansion Preserve Rank} Let $[\partial_d]$ be the compatible matrix associated to $\partial_d$. Then the isotropy expansion of $[\partial_d]$ has the same rank as $[\partial_d]$.
\end{lemma}

\begin{proof}

We prove the result by introducing an intermediate matrix $\IsotExpRow{d}$ that has the same rank as both $[\partial_d]$ and $\IsotExp{d}$. 
Let $\IsotRed{X}{d}$ and $\IsotRed{X}{d-1}$ be the $d$-th and $(d-1)$-th isotropy index-reducing functions of $X$. We then define $\IsotExpRow{d}$ to be the $(|\Sigma_{d-1}|\times n|G|)$ matrix defined as follows:
$$\IsotExpRow{d}_{i',j} := [\partial_d]_{i',\IsotRed{X}{d-1}(j)}$$
for $i'= 1,\dots, |\Sigma_{d-1}|$ and $j = 1,\dots,n|G|$.

Recall from \cref{Prop: Range of Isotrop Index Reducing} that the function $\IsotRed{X}{d-1}$ is surjective. Hence for each column $j'$ of $[\partial_d]$, there exists $j$ such that $\IsotRed{X}{d-1}(j) = j'$, and column $j$ of $\IsotExpRow{d}$ is the same as column $j'$ of $[\partial_d]$. Conversely, each column $j$ of $\IsotExpRow{d}$ is the same as column $\IsotRed{X}{d-1}(j)$ of $[\partial_d]$. Hence $\IsotExpRow{d}$ has the same set of columns as $[\partial_d]$  (with some columns duplicated), and thus $\rank \IsotExpRow{d} = \rank [\partial_d]$.

Then, by comparing the definitions, for $i= 1,\dots, |\Sigma'_{d-1}||G|$ and $j = 1,\dots,n|G|$,
$$\IsotExp{d}_{i,j} =\IsotExpRow{d}_{\IsotRed{X}{d-1}(i),j}\,.$$

We then see that, similar to the relation
between $\IsotExpRow{d}$ and $[\partial_d]$, by the surjectivity of $\IsotRed{X}{d-1}$, $\IsotExp{d}$ has the same set of rows as $\IsotExpRow{d}$ (with some rows duplicated), and thus $\rank \IsotExp{d} = \rank \IsotExpRow{d} = \rank [\partial_d]$ as desired.
\end{proof}

Note that \cref{Lemma: Isotrope Expansion Preserve Rank} implies that the rank of $\IsotExp{d}$ is independent of the choice of lift as well as of the ordered simplicial basis.  However, the construction of $\IsotExp{d}$ deeply depends on these choices.

Lastly we prove a lemma describing the entries of the isotropy expansion.

\begin{lemma}\label{Lemma: Entries of Isotrope Expansion} Let $[\partial'_d]$ be the matrix associated to $\partial'_d$ and $\IsotExp{d}$ the isotropy expansion of the compatible matrix $[\partial_d]$. Then for $a=1,\dots,m$, $b=1,\dots, n$, and $c,c' = 1,\dots,k$, we have the following equation: 
\begin{align*}
\IsotExp{d}_{ka-k+c,kb-k+c'} &=
  \begin{cases}
   [\partial'_d]_{a,b}        & \text{if } g_c\ell(\omega'_a)\subseteq g_{c'}\ell(\psi'_b) \\
   0        & \text{otherwise}
  \end{cases}
\end{align*}
\end{lemma}

\begin{proof}
Recall from \cref{Def: Isotrope Expansion} that
$$\IsotExp{d}_{r,s} = [\partial_d]_{\IsotRed{X}{d-1}(r),\IsotRed{X}{d}(s)}$$
for $r= 1,\dots, mk$ and $s = 1,\dots,nk$, where $\IsotRed{X}{d},\IsotRed{X}{d-1}$ are respectively the $d$-th and $(d-1)$-th isotropy index-reducing function of $X$ corresponding to $\ell$, the ordering $(g_1,\dots,g_{k})$ of $G$, and the ordering of $\Sigma'_d$ and 
$\Sigma'_{d-1}$ induced by $([\psi'_1],\dots,[\psi'_{n}])$ and $([\omega'_1],\dots,[\omega'_{m}])$. 

For $a= 1,\dots, m$ and $b=1,\dots, n$, let $p_a$ and $q_b$ be the unique integers such that $\psi_{q_b} = \ell(\psi'_b)$ and $\omega_{p_a} = \ell(\omega'_a)$. Then by \cref{indexreducing}, for $c,c' = 1,\dots, k$, $\IsotRed{X}{d}(kb-k+c') = q_b-1+j$ where $j$ satisfies $(gG_{\ell(\psi'_b)})_j = g_{c'}G_{\ell(\psi'_b)}$, while $\IsotRed{X}{d-1}(ka-k+c) = p_a-1+i$ where $i$ satisfies $(gG_{\ell(\omega'_a)})_i = g_cG_{\ell(\omega'_a)}$. In particular, $\IsotExp{d}_{ka-k+c,kb-k+c'} = [\partial_d]_{\IsotRed{X}{d-1}(ka-k+c),\IsotRed{X}{d}(kb-k+c')} = [\partial_d]_{p_a-1+i,q_b-1+j}$. We will show:
\begin{align*}
[\partial_d]_{p_a-1+i,q_b-1+j} &=
  \begin{cases}
   [\partial'_d]_{a,b}        & \text{if } g_c\ell(\omega'_a)\subseteq g_{c'}\ell(\psi'_b)\,, \\
   0        & \text{otherwise}\,.
  \end{cases}
\end{align*}

Note that as ordering indices of elements in $G/G_{\ell(\omega'_a)}$ and $G/G_{\ell(\psi'_b)}$ respectively, $1 \leq i \leq |G/G_{\ell(\omega'_a)}|$ and $1 \leq j \leq  |G/G_{\ell(\psi'_b)}|$. By our construction of the compatible orderings of $\Sigma_d$ and $\Sigma_{d-1}$ from \cref{Def: Ordering of Simplices compatible with quotient/lift/group}, $\omega_{p_a-1+i}$ is always in the orbit of $\omega_{p_a}$, while $\psi_{q_b-1+j}$ is always in the orbit of $\psi_{q_b}$. More specifically, 
$\omega_{p_a-1+i} = (gG_{\omega_{p_a}})_i(\omega_{p_a}) = g_cG_{\ell(\omega'_a)}\ell(\omega'_a) = g_c\ell(\omega'_a)$, 
while 
$\psi_{q_b-1+j} = (gG_{\psi_{q_b}})_j(\psi_{q_b}) = g_{c'}G_{\ell(\psi'_b)}\ell(\psi'_b) = g_{c'}\ell(\psi'_b)$.

We first consider the case $g_c\ell(\omega'_a) \subseteq g_{c'}\ell(\psi'_b)$, which is equivalent to $\omega_{p_a-1+i}\subseteq\psi_{q_b-1+j}$. From previous paragraph, 
$\pi(\omega_{p_a-1+i})= \pi(\omega_{p_a}) = \pi(\ell(\omega'_a)) = \omega'_a$,
while 
$\pi(\psi_{q_b-1+j})= \pi(\psi_{q_b}) = \pi(\ell(\psi'_b)) = \psi'_b$. 
Hence by \cref{Lemma: Entries of Compatible Matrix},
$\IsotExp{d}_{ka-k+c,kb-k+c'} = [\partial_d]_{p_a-1+i,q_b-1+j} = [\partial'_d]_{a,b}$. 

We now consider the case $g_c\ell(\omega'_a) \nsubseteq g_{c'}\ell(\psi'_b)$.
Recall that $[\partial_d]$ is the matrix associated to $\partial_d$ with respect to the compatible ordered basis $([\psi_1],\dots,[\psi_{|\Sigma_d|}])$ and $([\omega_1],\dots,[\omega_{|\Sigma_{d-1}|}])$.
Note that if $\omega_{p_a-1+i} =  g_c\ell(\omega'_a) \nsubseteq g_{c'}\ell(\psi'_b) =\psi_{q_b-1+j}$, then $\partial_d([\psi_{q_b-1+j}])$ has no $[\omega_{p_a-1+i}]$ components, so $\IsotExp{d}_{ka-k+c,kb-k+c'} = [\partial_d]_{p_a-1+i,q_b-1+j}=0$.
\end{proof}

\section{\texorpdfstring{$G$}{G}-Boundary Matrix}\label{sec: G-Boundary Matrix}

\subsection{Isotropy Transfer Triple }\label{sec: Modified Algorithm}

In this section we will modify the framework of complexes of group to encode the algebraic information needed to compute homology. Our reformulation is based on the associated output triple recalled from \cite{CarboneNandaNaqvi}  in \cref{Def: Associated Complex and Morphism}.

Aside from the regular $G$-complex, an associated output triple also depends on an arbitrary lift and an arbitrary transfer. Our modification in \cref{Def: Associated Reduced Morphism} eliminates the transfer by replacing it with an extended transfer that is uniquely determined by the input complex and the lift. We also extend the domain so that the extended transfer can be defined on any pair of simplices, not just those that share a face relation. Finally the extended transfer can also be presented in a matrix format which will be used to define the $G$-boundary matrix in \cref{Def: G-Boundary Matrix}.

Given a regular $G$-complex, a lift $\ell$, and simplices $\omega',\psi'$ of $X$ such that $\omega'\subseteq \psi'$. By definition of a lift, $\pi(\ell(\psi')) =\psi'$. Since $\omega'\subseteq \psi'$, \cref{Prop: Unique simplex implied by face in quotient} implies that there exists a unique $\omega$ such that $\omega\subseteq \ell(\psi')$ and $\pi(\omega) =\omega'$. With this in mind, we define the following:

\begin{definition}\label{Def: Associated Reduced Morphism}Let $G$ be a finite group. An \textbf{isotropy transfer triple associated to $X$ and $\ell$} is the triple $(X/G,S_\ell,T_\ell^*)$ defined as follows: 
\begin{itemize}
\item[(1)] $Y = X/G$.
\item[(2)] $S: \Sigma' \rightarrow \{\text{subgroups of $G$}\}$ is a function, where for any simplex $\psi'$ of $X/G$, $S_\ell(\psi')= G_{\ell(\psi')}$.
\item [(3)] $T^*: \{(\psi',\omega') \mid \text{$\psi',\omega'\in \Sigma'$}, \dim(\psi')=\dim(\omega')+1\}\ \rightarrow \{\text{subset of $G$}\}$ is a function, where for any simplices $\psi',\omega'$ of $X/G$, $
T_\ell^*(\psi',\omega') = \begin{cases}
        \{g\in G\mid\omega = g(\ell(\omega'))\} & \text{if $\omega'\subseteq \psi'$,}
        \\
        \emptyset & \text{if $\omega'\nsubseteq \psi'$}.
        \end{cases}
$
\end{itemize}

\end{definition}

We call $S_\ell$ the \textbf{isotropy subgroup function} and $T_\ell^*$ the \textbf{extended transfer}. There also is an alternative description of $T_\ell^*$.

\begin{proposition}\label{Prop: Alternative definition of extended transfer} Let $X$ be a regular $G$-complex and $(X/G,S_\ell,T_\ell^*)$ the associated isotropy transfer triple. Then $T_\ell^*(\psi',\omega') = \{g\in G\mid g(\ell(\omega'))\subseteq \ell(\psi')\}$.
\end{proposition}

\begin{proof}

When $\omega'\subseteq \psi'$, by \cref{Prop: Unique simplex implied by face in quotient} there exists a unique face $\omega$ of $\ell(\psi')$ such that $\pi(\omega) =\omega'$, i.e. $\omega$ is the unique simplex that is a face of $\ell(\psi')$ and is also in the orbit of $\ell(\omega')$. Hence for all $g\in G$, $g(\ell(\omega'))\subseteq \ell(\psi')$ if and only if $g(\ell(\omega')) = \omega$, and $T_\ell^*(\psi',\omega') = \{g\in G\mid\omega = g(\ell(\omega'))\} = \{g\in G\mid g(\ell(\omega'))\subseteq \ell(\psi')\}$ as desired.

When $\omega' \nsubseteq\psi'$, we claim $\{g\in G\mid g(\ell(\omega'))\subseteq \ell(\psi')\}$ is empty. If not, since the quotient simplicial map $\pi$ preserves face relations, the existence of some $g\in G$ such that $g(\ell(\omega'))\subseteq \ell(\psi')$ implies $\omega' = \pi(\ell(\omega')) = \pi(g(\ell(\omega'))) \subseteq \pi (\ell(\psi')) = \psi'$. This is a contradiction.

Hence in either case the equation holds.
\end{proof}

Recall from \cref{Def: Associated Complex and Morphism} and \cref{Def: Associated Reduced Morphism} that $S_\ell(\omega')$ is the isotropy subgroup of $\ell(\omega')$. We introduce one more key property of $T_\ell^*$.

\begin{lemma}\label{Lemma: T* is coset}Let $X$ be a regular $G$-complex and $(X/G,S_\ell,T_\ell^*)$ the associated isotropy transfer triple. Let $\psi',\omega'$ be a pair of simplices of $X/G$ such that $\omega' \subseteq \psi'$ and $\dim \psi' = \dim \omega'+1$. Then $T_\ell^*(\psi',\omega')$ is non-empty and is a left coset of $S_\ell(\omega')$ in $G$.

\end{lemma}

\begin{proof}

By definition $T_\ell^*(\psi',\omega') = \{g\in G\mid\omega = g(\ell(\omega'))\}$, where $\omega$ is the unique face of $\ell(\psi')$ such that $\pi(\omega) = \omega'$. To see that $\{g\in G\mid\omega = g(\ell(\omega'))\}$ is non-empty, note that $\pi(\omega) = \omega' =\pi(\ell(\omega'))$ by \cref{Def:Lift}. Hence $\omega$ and $\ell(\omega')$ are in the same orbit, so there exists some $g\in G$ such that $\omega= g(\ell(\pi(\omega)))$.

To see that $T_\ell^*(\psi',\omega')$ is a left coset of $S_\ell(\omega')$, we first recall from \cref{Def: Associated Reduced Morphism} that $S_\ell(\omega')$ is the isotropy subgroup of $\ell(\omega')$. Let $g_0$ be an arbitrary element of $T_\ell^*(\psi',\omega')$.  We claim that $T_\ell^*(\psi',\omega') = g_0 S_\ell(\omega')$. For any $g \in T_\ell(\psi',\omega')$, $g(\ell(\omega')) = \omega = g_0(\ell(\omega'))$, hence $g_0^{-1}g(\ell(\omega')) = \ell(\omega')$ and $g_0^{-1}g$ is in the isotropy subgroup $S_\ell(\omega')$ of $\ell(\omega')$. Therefore $g\in g_0 S_\ell(\omega')$. Conversely, for any $g\in g_0S_\ell(\omega')$, $g(\ell(\omega')) = g_0(\ell(\omega')) = \omega$, so $g\in T_\ell^*(\psi',\omega')$, as claimed.
\end{proof}

\subsection{\texorpdfstring{$G$}{G}-Boundary Matrix}

This section introduces the notion of a $G$-boundary matrix, which is our primary tool for reconstructing data about $X$ from $X/G.$ We start by introducing an auxiliary function that we compose with the extended transfer to produce an element of the group ring $\mathbb{F}G$. We will organize the output of this function into a matrix which we call a \emph{transfer matrix}. The $G$-boundary matrix will then  be defined as the entry-wise product of the transfer matrix with a suitable choice of boundary matrix of the quotient simplicial complex.

We begin with the definition of our auxiliary function.

\begin{definition} Let $\mathbb{F}$ be any field and $G$ be a finite group. We define the function $\bm{\sigma}_\mathbb{F}: \{\text{subsets of $G$}\}\rightarrow \mathbb{F}G$ by $A \mapsto \sum_{g\in A} g$.
\end{definition}

Given an isotropy transfer triple and a dimension $d$, we now define the corresponding transfer matrix.

\begin{definition}\label{Def: Transfer Matrix} Let $\mathbb{F}$ be a field, $X$ a regular $G$-complex, $\ell$ a lift, and $(X/G,S_\ell,T_\ell^*)$ the associated isotropy transfer triple. Let $(\psi'_1,\dots,\psi'_{n})$ and $(\omega'_1,\dots,\omega'_{m})$ respectively be orderings of the $d$-simplices and $(d-1)$-simplices of $X/G$. The \textbf{$d$-th transfer matrix in $\mathbb{F}$ (corresponding to $X$, $\ell$, and the orderings of the simplices of $X/G$)} is the matrix $[T_d]$ with entries in $\mathbb{F}G$ where $[T_d]_{i,j} :=\bm{\sigma}_\mathbb{F} (T_\ell^*(\psi'_i,\omega'_j))$. 

\end{definition}

Note that while $[T_d]$ depends on $\mathbb{F}$ and the orderings of the simplices, we have omitted them from the notation, as they will be clear from context.

We now construct a $G$-boundary matrix. Recall that an ordered simplicial basis of a $d$-th chain group automatically orders and orients all the $d$-simplices of a simplicial complex.

\begin{definition}\label{Def: G-Boundary Matrix} Let $\mathbb{F}$ be a field, $X$ a regular $G$-complex, $\ell$ a lift, and $(X/G,S_\ell,T_\ell^*)$ the associated isotropy transfer triple. Let $(\psi'_1,\dots,\psi'_{n})$ and $(\omega'_1,\dots,\omega'_{m})$ respectively be orderings of the $d$-simplices and $(d-1)$-simplices of $X/G$, and $([\psi'_1],\dots,[\psi'_{n}])$ and $([\omega'_1],\dots,[\omega'_{m}])$ the corresponding ordered simplicial bases of $C_d(X/G;\mathbb{F})$ and $C_{d-1}(X/G;\mathbb{F})$ respectively. Let $[T_d]$ be the corresponding $d$-th transfer matrix. Then the \textbf{$d$-th $G$-boundary matrix $\GBoundary{d}$ (corresponding to $X$, $\ell$, $([\psi'_1],\dots,[\psi'_{n}])$, and $([\omega'_1],\dots,[\omega'_{m}])$)} is the $(|\Sigma'_{d-1}|\times |\Sigma'_d|)$ matrix, where $\GBoundary{d}_{a,b} := [\partial'_d]_{a,b}[T_d]_{a,b}$.

\end{definition}

\section{Reconstructing homology of \texorpdfstring{$\mathbb{Z}_k$}{Zk}-complexes}
\label{sec: Rank Algorithm}

In this section we state and prove our main result, which completes the construction of a chain complex quasi-isomorphic to the standard simplicial chain complex for $X$ in the setting where $G = \mathbb{Z}_k$. First, we construct a function sending matrices over the group ring $\mathbb{F}G$ to matrices over the field $\mathbb{F}$. We use this function to relate the $G$-boundary matrix to the compatible isotropy expansion of the boundary matrix. Then, we demonstrate that, using this relationship, we can reconstruct the homology of $X$ from the chain complex of $X/G.$

Given a representation $\rho:G\rightarrow M(k,\mathbb{F})$, by considering $M(k,\mathbb{F})$ as an $\mathbb{F}$-linear space, we can extend $\rho$ linearly to the entire $\mathbb{F}G$. We will call such an extended function a \textbf{representation of $\mathbb{F}G$} and abuse notation to also denote it by $\rho$.

We also introduce our notation for the matrix blocks of interest. 

\begin{remark}\label{Rmk: Matrix Block and Block Indices}
Let $M$ be an $(mk \times nk)$-matrix. For any integer $a$, we denote the set of consecutive integers $\{(k-1)a+1, \dots, ka\}$ by $L_a$. We then denote by $M(a, b)$ the submatrix of $M$ whose rows and columns are indexed by $L_a$ and $L_b$, respectively. Note that $\{L_a\}_{a=1,\dots,m}$ and $\{L_b\}_{b=1,\dots,n}$ partition $\{1, \dots, mk\}$ and $\{1, \dots, nk\}$, respectively, and the block matrices $\{M(a, b)\}_{a=1,\dots,m; b=1,\dots,n}$ partition $M$. Thus, determining $M(a, b)$ for every $a$ and $b$ is equivalent to determining the matrix $M$.
\end{remark}

We then extend a representation of $\mathbb{F}G$ to any matrix over $\mathbb{F}G$.

\begin{definition}\label{Def: Entry by Entry Representation of Matrix}
Let $G$ be an ordered finite group, $|G|=k$, $\mathbb{F}$ a field, and $\rho:\mathbb{F}G\rightarrow M(k,\mathbb{F})$ a representation of $\mathbb{F}G$. For positive integers $m,n$, let $A$ be an $(m\times n)$ matrix with entries in $\mathbb{F}G$. Then define the function $[\rho]_{mn}:M(m\times n, \mathbb{F}G)\rightarrow M(mk\times nk, \mathbb{F})$ by applying $\rho$ element-wise; that is, $A\mapsto [\rho]_{mn} A$ where $[\rho]_{mn} A$ is the $(mk\times nk)$ matrix defined by $([\rho]_{mn} A)(a,b) := \rho(A_{a,b}),$ for $a=1\dots, m$ and $b=1,\dots,n$. We call $[\rho]_{mn}$ the \textbf{matrix extension} of the representation $\rho$. 
\end{definition}

When the context is clear we  omit the subscripts, writing $[\rho]$ for $[\rho]_{mn}$. We observe some additional properties of the matrix form of a representation.

\begin{proposition}\label{Proposition: Entry-by-Entry Rep is Ring Hom} Let $\rho: G \to M(k, \mathbb{F})$ be a representation of $\mathbb{F}G$. Then $[\rho]$ preserves matrix addition and multiplication. Moreover, if $\rho$ is injective, then $[\rho]$ is also injective. In this case, $[\rho]$ maps invertible matrices to invertible matrices, that is,  $[\rho]$ preserves inverses.

\end{proposition}

\begin{proof}

The fact that $[\rho]$ preserves matrix addition and matrix multiplication, and is injective when $\rho$ is injective, follows directly from the entry-wise application of $\rho.$

To show that $[\rho]$ preserves inverses, let $A$ be an $(n\times n)$ invertible matrix. Then there exists $A^{-1}$ such that $AA^{-1}$ is the identity matrix in $M(n, \mathbb{F}G)$. Note that multiplicative identity of $\mathbb{F}G$ is $1_\mathbb{F}\,e$, where $1_\mathbb{F}$ and $e$ are the (multiplicative) identity elements of $\mathbb{F}$ and $G$ respectively.

Since $\rho$ is an injective ring homomorphism, it maps the identity to the identity, and $1_\mathbb{F}\,e$ is mapped to the identity matrix $I_k$ of $M(k,\mathbb{F})$. Then $[\rho](AA^{-1})$ is the block diagonal $(nk\times nk)$ matrix, whose diagonal blocks are all $I_k$. In other words $[\rho](AA^{-1})$ is the identity matrix $I_{nk}$. Then $([\rho]A)([\rho]A^{-1}) = [\rho](AA^{-1}) = I_{nk}$, hence $[\rho]A$ is invertible and $([\rho]A)^{-1} = [\rho](A^{-1})$.
\end{proof}

We now consider the special case where $G$ is a finite cyclic group $\mathbb{Z}_k$. We will construct a representation $\rho_\alpha$ of $\mathbb{F}\mathbb{Z}_k$ using the observations in \cref{subsec: Circulant Matrix and Group Rings}, and then applying \cref{Def: Entry by Entry Representation of Matrix} to $\rho_\alpha$.

We first introduce some notation. Fix a generator $\alpha\in \mathbb{Z}_k$ and use it to construct an ordering $(e,\alpha, \dots , \alpha^{k-1})$ of $\mathbb{Z}_k$. By \cref{Prop: Cyclic gives Circulant} each element $w = \sum_{i=0}^{n-1} a_i \alpha^i \in \mathbb{F}\mathbb{Z}_k$ corresponds to a circulant matrix $M(\mathbb{F}\mathbb{Z}_k, w)$. Then, we define the following.

\begin{definition}\label{Def: Representation of FZk} Let $\mathbb{F}$ be a field and $\alpha$ a generator of $\mathbb{Z}_k$. The \textbf{representation of $\mathbb{Z}_k$ corresponding to $\alpha$} is the function $\rho_{\alpha}: \mathbb{F}\mathbb{Z}_k\rightarrow M(n,\mathbb{F})$ given by $w\mapsto M(\mathbb{F}\mathbb{Z}_p,w)^T$.

\end{definition}
In particular, for any $w =  \sum_{i=0}^{n-1} a_{i}\alpha^i$, $\rho_\alpha(w) = \CircMatrix{(a_0,a_1,\dots,a_{n-1})}^T$.

To see that $\rho_\alpha$ is indeed a representation of $\mathbb{Z}_k$, we first apply \cref{Thm: RG embeds to Matrix} to the ordering $(e,\alpha, \dots , \alpha^{k-1})$ of $\mathbb{Z}_k$ to obtain ring homomorphism $\rho: w\mapsto M(\mathbb{F}\mathbb{Z}_k,w)$. As $\mathbb{F}\mathbb{Z}_k$ is commutative,  for any $w_1,w_2\in \mathbb{F}\mathbb{Z}_k$,
\begin{align*}\rho_\alpha(w_1 w_2) &= \rho_\alpha(w_2 w_1) = M(\mathbb{F}\mathbb{Z}_k,w_2 w_1)^T
= (\rho(w_2w_1))^T\\& = (\rho(w_2)\rho(w_1))^{T}
= \rho(w_2)^T\rho(w_1)^{T} = \rho_\alpha(w_1)\rho_\alpha(w_2).\end{align*}
As the transpose of a circulant matrix is also circulant, the image of $\rho_\alpha$ is still the set of all $\mathbb{F}\mathbb{Z}_k$-matrices.

We will use the above defined representation of $\mathbb{Z}_k$ corresponding to $\alpha$ for the remainder of this article. Applying \cref{Def: Entry by Entry Representation of Matrix} to $\rho_\alpha$ gives us a map $[\rho_\alpha]$ from matrices over $\mathbb{F}\mathbb{Z}_k$ to larger matrices over $\mathbb{F}$. We will use $[\rho_\alpha]$ to relate the $G$-boundary matrix and the compatible isotropy expansion of the boundary matrix. 

\begin{lemma}
\label{Lemma: Expansion equals Representation}
Let $X$ be a regular $\mathbb{Z}_k$-complex and $\ell$ a lift. Let $([\psi'_1],\dots,[\psi'_{n}])$ and $([\omega'_1],\dots,[\omega'_{m}])$ be respectively an ordered simplicial basis of $C_d(X/\mathbb{Z}_k;\mathbb{F})$ and $C_{d-1}(X/\mathbb{Z}_k;\mathbb{F})$. Let $\rho_\alpha$ be the representation of $\mathbb{Z}_k$ corresponding to $\alpha$. Let $\GBoundary{d}$ be the corresponding $\mathbb{Z}_k$-boundary matrix and $\IsotExp{d}$ the compatible isotropy expansion of $[\partial_d]$. Then $$\IsotExp{d} = [\rho_\alpha]\GBoundary{d}\,.$$
\end{lemma}

\begin{proof}

Let $\IsotRed{X}{d}$ and $\IsotRed{X}{d-1}$ respectively be the $d$-th and $(d-1)$-th corresponding isotropy index reducing function of $X$ from \cref{Def: Isotrope Index Reducing Function}. Let $(X/\mathbb{Z}_k,S_\ell,T_\ell^*)$ be the associated isotropy transfer triple of \cref{Def: Associated Reduced Morphism}. Finally, let $\{\IndicesOfOrbit{X}{d-1}{a}\}_{a=1,\dots,m}$ and $\{\IndicesOfOrbit{X}{d}{b}\}_{b=1,\dots,n}$ respectively be the corresponding $(d-1)$-th and $d$-th lifted partition from \cref{Def: Partition Index by Orbits}.

Recall that $\IsotExp{d}$ and $[\rho_\alpha]\GBoundary{d}$ are both $(mk\times nk)$ matrices. Following \cref{Rmk: Matrix Block and Block Indices}, it suffices to prove that $\IsotExp{d}(a,b) = ([\rho_\alpha]\GBoundary{d})(a,b)$ for every $a =1,\dots,m$ and $b=1,\dots,n$. By \cref{Def: Entry by Entry Representation of Matrix}, $([\rho_\alpha][\partial'_d]^{S_\ell,T})(a,b) = \rho_\alpha([\partial'_d]^{S_\ell,T}_{a,b})$. Hence we have reduced our goal to proving $\IsotExp{d}(a,b) = \rho_\alpha([\partial'_d]^{S_\ell,T}_{a,b})\, .$

We claim that 
$\IsotExp{d}(a,b) = \CircMatrix{(v_1,\dots,v_k)}^T$,
where $v_c = [\partial'_d]_{a,b}$ if $\alpha^{c-1} \in T_\ell^*(\psi'_b, \omega'_a)$, and $v_c = 0$ otherwise. For simplicity of notation, we will take the subscript index of $v$ modulo $k$, due to the fact that $\alpha^k = e$. We will prove our claim by demonstrating that each column of the two matrices is identical. Recall from \cref{Rmk: Matrix Block and Block Indices} that $L_a=\{(k-1)a+1, \dots, ka\}$ for any integer $a$. 
By considering the indices in the sets $L_a$ and $L_b$, for $c,c' = 1,\dots,k$, we obtain that $(\IsotExp{d}(a,b))_{c,c'} = \IsotExp{d}_{ka-k+c,kb-k+c'}$. On the other hand, by \cref{Def: Circulant Matrix}, the $c'$-th column of $\CircMatrix{v_1,\dots,v_k}^T$ is $(\mathcal{S}^{c'-1}(v_1, \dots, v_k))^T$, where $\mathcal{S}$ is the shift operator of \cref{Def: Circulant Matrix}.

For $c' = 1,\dots, k$, consider the $c'$-th column $(\IsotExp{d}_{ka-k+1,kb-k+c'},\dots,\IsotExp{d}_{ka,kb-k+c'})^T$ of $\IsotExp{d}(a,b)$. By \cref{Lemma: Entries of Isotrope Expansion},
$\IsotExp{d}_{ka-k+c,kb-k+c'} = [\partial'_d]_{a,b}$
whenever 
$\alpha^{c-1}\ell(\omega'_a) = g_c\ell(\omega'_a) \subseteq g_{c'}\ell(\psi'_b) = \alpha^{c'-1}\ell(\psi'_b)$,
and is $0$ otherwise. By \cref{Prop: Alternative definition of extended transfer}, $T_\ell^*(\psi'_b,\omega'_a) = \{g\in G\mid g(\ell(\omega'_a))\subseteq \ell(\psi'_b)\}$. Hence $\alpha^{c-1}\ell(\omega'_a) \subseteq \alpha^{c'-1}\ell(\psi'_b)$ if and only if
$ \alpha^{c-c'}\ell(\omega')\subseteq \ell(\psi'_b)$ if and only if $\alpha^{c-c'}\in T_\ell^*(\psi'_b,\omega'_a)$. 

Therefore, when \(\alpha^{c-1}\ell(\omega'_a) \subseteq \alpha^{c'-1}\ell(\psi'_b)\), $\alpha^{c-c'}\in T_\ell^*(\psi'_b,\omega'_a)$ and \(\IsotExp{d}_{ka-k+c,kb-k+c'} = [\partial'_d]_{a,b} = v_{c-c'+1}\). When \(\alpha^{c}\ell(\omega'_a) \nsubseteq \alpha^{c'}\ell(\psi'_b)\), $\alpha^{c-c'}\notin T_\ell^*(\psi'_b,\omega'_a)$ and \(\IsotExp{d}_{ka-k+c,kb-k+c'} = 0 = v_{c-c'+1}\). In both cases, \(\IsotExp{d}_{ka-k+c,kb-k+c'} = v_{c-c'+1}\). 
Consequently, the \(c'\)-th column of \(\IsotExp{d}(a,b)\) is given by \((v_{2-c'}, v_{3-c'}, \dots, v_{k+1-c'})^T = (\mathcal{S}^{c'-1}(v_1, \dots, v_k))^T\). This shows that \(\IsotExp{d}(a,b) = \CircMatrix{(v_1, \dots, v_k)}^T\), as we claimed.

It then suffices to prove $\rho_\alpha([\partial'_d]^{S_\ell,T}_{a,b}) = \CircMatrix{(v_1,\dots,v_k)}^T$ also.  \cref{Def: G-Boundary Matrix} and \cref{Def: Transfer Matrix} imply that 
$[\partial'_d]^{S_\ell,T}_{a,b}= [\partial'_d]_{a,b}[T_d]_{a,b}$,
where
$[T_d]_{a,b} = \sum_{g\in T_\ell^*(\psi'_b,\omega'_a)} g = \sum_{\alpha^{c-1}\in T_\ell^*(\psi'_b,\omega'_a)} \alpha^{c-1}.$
Since $\rho_\alpha$ is a ring homomorphism, we then obtain that
$\rho_\alpha([\partial'_d]^{S_\ell,T}_{a,b}) = [\partial'_d]_{a,b}\sum_{\alpha^{c-1}\in T_\ell^*(\psi'_b,\omega'_a)} \rho_{\alpha}(\alpha^{c-1}) $. 
Note that $\rho_{\alpha}(\alpha^{c-1}) = \CircMatrix{e_c}^T,$ where $\{e_c\}_{c=1,\dots k}$ is the standard basis of $\mathbb{F}^k$. Finally, linearity of the map $\CircMatrix{\cdot}$ and our definition of $v_c$ yield 
\begin{align*}    
\rho_\alpha([\partial'_d]^{S_\ell,T}_{a,b}) 
&= {[\partial'_d]_{a,b}}\sum\limits_{\alpha^{c-1}\in T_\ell^*(\psi'_b,\omega'_a)} \,\CircMatrix{e_c}^T\\
&= \sum\limits_{\alpha^{c-1}\in T_\ell^*(\psi'_b,\omega'_a)} {[\partial'_d]_{a,b}}\,\CircMatrix{e_c}^T 
+ \sum\limits_{\alpha^{c-1}\notin T_\ell^*(\psi'_b,\omega'_a)} 0\,\CircMatrix{e_c}^T\\
&= \sum\limits_{\alpha^{c-1}\in T_\ell^*(\psi'_b,\omega'_a)} v_c\,\CircMatrix{e_c}^T
+ \sum\limits_{\alpha^{c-1}\notin T_\ell^*(\psi'_b,\omega'_a)} v_c\,\CircMatrix{e_c}^T\\
&= \sum\limits_{c=1,\dots, k} v_c\,\CircMatrix{e_c}^T
= \CircMatrix{v_1,\dots,v_k}^T =\IsotExp{d}(a,b),\end{align*}
as desired.
\end{proof}

\medskip

We now consider an $(m\times n)$ matrix, $M$, over a ring $R$. We say that \textbf{$M$ has a Smith normal form} $D$ if there exist invertible matrices $P$ and $Q$ such that $PMQ = D$, where $D$ is a diagonal matrix with diagonal entries $(a_1, \dots, a_r, 0, \dots, 0)$, with $a_i \in R$ satisfying $a_i \mid a_{i+1}$ for $1 \leq i \leq r-1$. The following result is summarized from \cite{SNFInCombinatorics}.

\begin{theorem}\label{Thm: Principal Ideal Ring has SNF} Let $M$ be a matrix over a ring $R$.  If every ideal of $R$ is principal, then $M$ has a Smith normal form.
\end{theorem}

In particular, if $\mathbb{F}$ is a field and $\mathbb{Z}_k$ a finite cyclic group, then any matrix over $\mathbb{F}\mathbb{Z}_k$ has a Smith normal form. This follows from the fact that $\mathbb{F}[x]$ is a principal ideal ring and $\mathbb{F}[x]\rightarrow \mathbb{F}[x]/\langle x^k-1\rangle \cong \mathbb{F}\mathbb{Z}_k$ is a quotient map. 

We can now state and prove our main theorem.

\begin{theorem}\label{Thm: Main} Let $\mathbb{F}$ be any field, $X$ be a regular $\mathbb{Z}_k$-complex, $\ell$ be any lift to $X$, and $d$ a non-negative integer. Let $\partial_d$ be the $d$-th boundary map in the chain complex $C_\ast(X; \mathbb{F}).$ Let $([\psi'_0],\dots,[\psi'_{n}])$ and $([\omega'_0],\dots,[\omega'_{m}])$ respectively be any $d$-th and a $(d-1)$-th ordered simplicial basis of $X/\mathbb{Z}_k$ and $[\partial'_d]^{S_\ell,T}$ be the $d$-th $G$-boundary matrix $[\partial_d']^{S_\ell,T}$ corresponding to $X$, $\ell$, $([\psi'_1],\dots,[\psi'_{n}])$, and $([\omega'_0],\dots,[\omega'_{m}]).$ Let $D$ be a smith normal form of $[\partial'_d]^{S_\ell,T}$. Let $\alpha$ be any generator of $\mathbb{Z}_k$ and $\rho_\alpha$ the representation of $\mathbb{Z}_k$ corresponding to $\alpha$. Then $\rank (\partial_d) = \sum\limits_{i} \rank (\rho_\alpha(D_{i,i}))$.
\end{theorem}

As computing homology over a field requires only the ranks of the boundary maps, \cref{Thm: Main} then immediately implies that we can reconstruct the vector space $H_d(X; \mathbb{F})$ from $X/G$ and an associated isotropy transfer triple. In particular, the reconstruction does not depend on the choice of the lift $\ell$ or the ordered simplicial basis $([\psi'_1],\dots,[\psi'_{n}])$ and $([\omega'_0],\dots,[\omega'_{m}])$. 

\begin{proof}

Let $[\partial_d]$ be the matrix compatible with $([\psi'_1],\dots,[\psi'_{n}])$, $([\omega'_0],\dots,[\omega'_{m}])$, $\ell$, and $(e,\alpha,\dots,\alpha^{k-1})$
(see \cref{Def: Associated Matrix Compatible with Quotient}). By \cref{Lemma: Isotrope Expansion Preserve Rank} and \cref{Lemma: Expansion equals Representation}, $\rank ([\partial_d]) = \rank(\IsotExp{d}) = \rank ([\rho_\alpha][\partial'_d]^{S_\ell,T})$. Hence it suffices to compute the rank of $[\rho_\alpha][\partial'_d]^{S_\ell,T}$, where $[\rho_\alpha]$ is as defined in  \cref{Def: Entry by Entry Representation of Matrix}.

By \cref{Proposition: Entry-by-Entry Rep is Ring Hom}, $[\rho_\alpha]$ preserves matrix multiplication. Moreover $\rho_\alpha$ is injective, since it is the composition of the matrix transpose and the embedding of \cref{Thm: RG embeds to Matrix}. Hence $[\rho_\alpha]$ is also injective and sends invertible matrices to invertible matrices. Let $D= P[\partial'_d]^{S_\ell,T}Q$ be a Smith normal form decomposition of $[\partial'_d]^{S_\ell,T}$. Then $[\rho_\alpha]D = ([\rho_\alpha]P)([\rho_\alpha][\partial'_d]^{S_\ell,T})([\rho_\alpha]Q)$. Since $[\rho_\alpha]P$ and $[\rho_\alpha]Q$ are both invertible, 
$\rank ([\rho_\alpha]D)= \rank ([\rho_\alpha][\partial'_d]^{S_\ell,T})$.

Observe that the $\rho_\alpha(D_{i,i})$ are the diagonal blocks of the block diagonal matrix $[\rho_\alpha]D$. Hence $\rank ([\rho]D) = \sum_i \rank \rho(d_i)$. Combining these observations,  we have $$\sum_i \rank \rho(d_i) = \rank ([\rho][\partial'_d]^{S_\ell,T}) = \rank ([\partial_d]) = \rank (\partial_d)$$ as desired.
\end{proof}

\section{Conclusion}

In this paper, we introduce a new approach to computing the homology of regular simplicial \( G \)-complexes with field coefficients by using their corresponding quotient spaces. To facilitate this approach, we expand upon the complex of groups framework introduced in \cite{BridsonHaefliger} and reformulated by \cite{CarboneNandaNaqvi}. Inspired by representation theory, we define \( G \)-boundary matrices in \cref{Def: G-Boundary Matrix}, which compress the algebraic information of the boundary maps of a $G$-complex. Restricting to \( G=\mathbb{Z}_k \), we ensure that a \( G \)-boundary matrix,  matrix over the group ring \( \mathbb{F}G \), admits a Smith normal form. In \cref{Thm: Main}, we demonstrate how to compute the ranks of the boundary maps of a \( G \)-complex from the Smith normal form of its \( G \)-boundary matrices. In principle, this approach can be applied in other contexts, however to our knowledge it remains an open problem for which fields \(\mathbb{F}\) and groups \(G\) the necessary Smith normal form to support these computations exists.

Our work presents an alternative approach to computing homology in the context of a group action that has the potential to be more efficient in terms of both time and memory by compressing the matrices that must be reduced. Additionally, we introduce a novel perspective by examining equivariant compression, previously combinatorially conceived through the theory of complexes of groups, algebraically through the lens of representation theory. We hope this will contribute to the development of a broader theoretical framework for homology computations in the context of group actions.

The tools and insights developed in this paper suggest a number of avenues for further work. Here we highlight a few that we find particularly interesting:

\begin{itemize}

\item Our work relies on the assumption that a simplicial \( G \)-complex is regular. While regularity can always be achieved by performing a barycentric subdivision twice on any simplicial \( G \)-complex, this process does not lend itself to efficient computation. This issue can be partially addressed by performing barycentric subdivision only on key simplices. Additionally, we are currently exploring ways to relax the current regularity condition to a weaker condition that can be achieved by performing barycentric subdivision just once, resulting in a less rigid combinatorial structure on the quotient complex.

\item We have computed the homology only in the case where $G$ is a finite cyclic group $\mathbb{Z}_k$. We first relied on the fact that $\mathbb{Z}_k$ is an abelian group. $\mathbb{F}\mathbb{Z}_k$ is then commutative, which we utilized to construct the necessary representation to relate a \( G \)-boundary matrix to the corresponding boundary map. If $G$ is non-abelian, we would need to modify the way we construct the representation. Secondly, we used the fact every matrix over $\mathbb{F}\mathbb{Z}_k$ admits Smith normal form. This is a significant challenge even for other finite abelian groups, as even for seemingly simple examples such as \( \mathbb{F} = \mathbb{F}_2 \) and \( G = \mathbb{Z}_2 \times \mathbb{Z}_2 \), matrices over \( \mathbb{F}G \) do not always admit a Smith normal form. This is interesting in contrast to the combinatorial case considered in \cite{CarboneNandaNaqvi}, where such considerations do not arise.

\end{itemize}

\bibliographystyle{amsplain}

\end{document}